\documentclass{amsart}
\usepackage[foot]{amsaddr}
\usepackage[english]{babel}

\usepackage[utf8]{inputenc}

\usepackage{verbatim}

\usepackage[numbers]{natbib}

\usepackage{amsmath, amssymb,latexsym,amsthm}

\usepackage{tikz}
\usetikzlibrary{decorations.pathreplacing}

\usepackage{enumitem}

\usepackage{hyperref}

\theoremstyle{definition}
\newtheorem{remark}{Remark}

\theoremstyle{plain}
\newtheorem{theorem}{Theorem}

\newtheorem{lemma}{Lemma}

\theoremstyle{definition}
\newtheorem{definition}{Definition}
\newtheorem{condition}{Condition}

\newcommand{\Fr}{\mathrm{F}}
\newcommand{\Rel}{\mathrm{Rel}}
\newcommand{\Mon}{\mathrm{Mon}}
\newcommand{\Ideal}{\mathcal{I}}
\newcommand{\M}{\Lambda}
\newcommand{\Mp}{\Lambda^{\prime}}
\newcommand{\Sp}{\mathcal{S}}
\newcommand{\Add}{\mathrm{Add}}
\newcommand{\Dp}{\mathrm{Dp}}
\newcommand{\Low}{\mathrm{L}}
\newcommand{\MUp}{\text{$\mathrm{Equal}$-$\mathrm{f}$}}
\newcommand{\mincov}{\mathsf{MinCov}}
\newcommand{\nvirt}{\mathsf{NVirt}}
\newcommand{\llangle}{\langle\langle }
\newcommand{\rrangle}{\rangle\rangle }
\newcommand{\DMUp}[1]{\MUp(#1)_d}
\newcommand{\deglex}{\mathrm{deglex}}

\newcommand{\abs}[1]{\vert #1 \vert}

\title{Small Cancellation Rings are non-amenable}

\author{A.\,Atkarskaya}
\address{Einstein Institute of Mathematics, The Hebrew University of Jerusalem, Givat Ram, 9190401 Jerusalem, Israel}
\email{atkarskaya.agatha@gmail.com}

\begin{document}
\begin{abstract}
We show that small cancellation rings defined in~\cite{AKPR2} under some natural extra restrictions are non-amenable and contain free associative algebra with two free generators.
\end{abstract}

\maketitle

\section{Introduction}

The small cancellation theory for groups is well known (see~\cite{LS}). In particular, it is known that finitely presented small cancellation groups turned out to be word hyperbolic (when every relation needs at least 7 pieces). In papers~\cite{AKPR1} and~\cite{AKPR2} we develop group-like small cancellation theory for associative rings. We introduce these rings by generators and defining relations, define small pieces and develop ring analogues of small cancellation conditions. The main result of~\cite{AKPR2} is that such rings are non-trivial. Moreover there exists Dehn's-type algorithm aimed to check whether a given element of a small cancellation ring is trivial or non-trivial (see~\cite{Ly}).

Now it seems quite important to study further properties of small cancellation rings. M.\,Gromov conjectured that small cancellation algebras are non-amenable (M.\,Gromov, private correspondence), see~\cite{B},~\cite{El},~\cite{Gr}, for the definition of amenable associative algebras. In this paper we prove this conjecture. The main result is stated in Theorem~\ref{sc_rings_non_amenable} and claims that under some natural conditions small cancellation rings are non-amenable.

One can easily check that the free algebra with at least two free generators is non-amenable using the Folner criterion as stated by Bartholdi,~\cite{B}. For groups it is known that if a given group contains the free group $\Fr_2$, then this group is non-amenable. It is known that small cancellation groups satisfying condition $C(4)$ together with $T(4)$, as well as satisfying condition $C(7)$ (classic or graphical) contain $\Fr_2$ (see~\cite{El-M},~\cite{Col},~\cite{Gru}). So these groups are non-amenable. For associative algebras a similar property does not hold. Indeed, it is well known that solvable groups are amenable. There exist solvable groups that contain the free semigroup with two free generators $\langle x, y\rangle$ (see~\cite{Bel}). Let $G$ be such group. By the result of L.\,Bartholdi,~\cite{B}, the group algebra $kG$ is amenable and by construction it contains free associative algebra $k\langle x, y\rangle$. Along with non-amenability, we prove in the paper that under some additional conditions small cancellation rings contain free associative algebra with two free generators, see Theorem~\ref{free_subelgebra_exists}.

Let us notice that amenability of monomial algebras was studied in~\cite{BellBeeri}, in particular it is proved there that non-amenable monomial algebras must contain noncommutative free subalgebras.

It is well known that infinite not virtually cyclic hyperbolic groups contain $\Fr_2$, and hence are non-amenable. We consider small cancellation rings as a step towards (not yet defined) hyperbolic rings, so their non-amenability confirms this direction.

We want to notice that the construction of small cancellation rings is quite technical, so we do not present it in all the details in this paper. Here we state only main facts about small cancellation rings and for the details we constantly refer to paper~\cite{AKPR2}.

\textbf{Acknowledgements.} The author thanks M.\,Gromov whose question about non-amenability of small cancellation rings gave raise to this research. The author is grateful to E.\,Plotkin and E.\,Rips for reading the manuscript.

This research is supported by the ISF fellowship.

\section{Introduction to group-like small cancellation conditions for rings}
Let $\Fr$ be a free group, $k$ be a field, and $k\Fr$ be the corresponding group algebra. The reduced representatives of the elements of $\Fr$ are called words or monomials. Their linear combinations with coefficients from $k$ are called polynomials. We denote coefficients from~$k$ by small Greek letters (unless otherwise is stated). The free generators of $\Fr$ and their inverses are called letters. Let $A, B$ be monomials in $\Fr$. We write the product of $A$ and $B$ as $A\cdot B$. However, if there is no cancellation in this product, we omit $\cdot$ and write it as $AB$. 

Given a word $A$, we denote by $\abs{A}$ the length of $A$ in letters.

Given a set $H \subseteq k\Fr$, we denote its linear span over the field $k$ by $\langle H \rangle$. We denote the ideal of $k\Fr$ generated by~$H$ by~$ \llangle H\rrangle$.

Let $\Rel \subseteq k\Fr$ be a set of polynomials and $\Ideal = \llangle\Rel\rrangle$ be an ideal of $k\Fr$ generated by $\Rel$ as an ideal. Let $\Mon$ be all monomials that participate in the polynomials from $\Rel$. Group-like small cancellation theory for rings states a list of conditions that guarantee non-triviality of the quotient ring $k\Fr / \Ideal$. In this section we give a short introduction to these conditions. For all the details see~\cite{AKPR2}, Section~$2$, and~\cite{AKPR3}.

\begin{condition}[Compatibility Axiom]
We require that $\Rel$ is closed under multiplication by elements from $k$ and under multiplication by letters that cancel with some monomial. For instance, if $a + xb \in \Rel$, where $a, b$ are monomials and $x$ is a single letter, then $x^{-1}\cdot a + b\in \Rel$ as well. This condition is called Compatibility Axiom. For the further axioms we assume that $\Rel$ satisfies this condition. In particular, this condition implies that $\Mon$ is closed under taking subwords.
\end{condition}

The main notion is a small piece.
\begin{definition}[small pieces]
\label{sp}
We assume that $\Rel$ satisfies Compatibility Axiom. Let $c \in \Mon$. Let $p, q \in \Rel$ be polynomials such that $c$ is a subword in a monomial in $p$ and in $q$. Namely, $p$ contains a monomial $\widehat{a}_1c\widehat{a}_2$ and $q$ contains $\widehat{b}_1c\widehat{b}_2$ ($\widehat{a}_1$, $\widehat{a}_2$, $\widehat{b}_1$, $\widehat{b}_2$ are allowed to be empty). If at least one of $\widehat{b}_1\cdot \widehat{a}_1^{-1}\cdot p$ and $p\cdot \widehat{a}_2^{-1}\cdot \widehat{b}_2$ is not contained in $\Rel$, then $c$ is called \emph{a small piece with respect to $\Rel$}. By definition, we require that the empty word $1$ is always a small piece. The set of all small pieces with respect to $\Rel$ is denoted by~$\Sp$.
\end{definition}

We fix a constant $\tau \geqslant 10$. The central axiom is the following
\begin{condition}[Small Cancellation Axiom]
\label{sc_ax}
Assume $q_1, \ldots, q_n \in \Rel$ and a linear combination $\sum_{l = 1}^{n} \gamma_l q_l$ is non-zero after additive cancellations. Then there exists a monomial $a$ in $\sum_{l= 1}^{n} \gamma_l q_l$ with a non-zero coefficient after additive cancellations such that either $a$ can not be represented as a product of small pieces or every such representation of $a$ contains at least $\tau + 1$ small pieces.
\end{condition}

There is also the third condition called Isolation Axiom. This condition is very technical, so we do not give it here. Roughly speaking, it guarantees that monomials that differ only by one small piece in the beginning or at the end can not participate in one polynomial in some closure of $\Rel$ called $\Add(\Rel)$ (see \cite{AKPR2}, Section~$2$, Conditions~3a and~3b, and Section~$10$, for the definition of $\Add(\Rel)$).

The main result of Group-like Small cancellation theory is as follows
\begin{theorem}
Let $\Rel \subseteq k\Fr$ satisfy group-like small cancellation conditions. Then the quotient ring $k\Fr/ \llangle \Rel\rrangle$ is non-trivial.
\end{theorem}

\section{Preliminaries}

\label{prel}
In this section we state notions needed to prove the main results, that is, Theorems~\ref{free_subelgebra_exists} and~\ref{sc_rings_non_amenable}.

In~\cite{AKPR2} we introduce a measure on $\Mon$ and denote it by $\M$ (aka $\M$-measure). By definition $\M(u) = t$ for $u \in \Mon$ if $u$ can be represented as a concatenation of $t$ small pieces and can not be represented as a concatenation of less number of small pieces. If $u$ can not be represented as a concatenation of small pieces at all, then we put $\M(u) = \infty$.

The triple that consists of a word $A$, its subword $u$ and a position of $u$ in $A$ is called an occurrence of $u$ in $A$. Further we are interested in occurrences of $u \in \Mon$.

If a given occurrence of an element of $\Mon$ in a word $A$ is not contained in any other occurrence of elements of $\Mon$, then it is called \emph{maximal}. In what follows we call them shortly \emph{maximal occurrences}, without mentioning $\Mon$. See~\cite{AKPR2}, Section~$3$, Definitions~$3.1$ and~$3.2$ for the detailed definition.

The set of all maximal occurrences in a word $A$ such that they have $\M$-measure $\geqslant \tau$ is called \emph{the chart of $A$}. If there are no such occurrences, we say that the chart of $A$ is empty. In the next statement instead of a chart we use a term \emph{virtual chart}. This notion is a special extension of the notion of the chart of $A$ and its every element has $\M$-measure $\geqslant \tau - 2$ (see~\cite{AKPR2}, Section~$6$, Definition~$6.5$).

\begin{lemma}
\label{empty_chart_words}
Let $\Rel \subseteq k\Fr$ satisfy group-like small cancellation conditions. Let $w_1, \ldots, w_k, \ldots$ be finite or infinite set of pairwise different words. Assume that $w_i$ does not contain occurrences of elements from $\Mon$ of $\M$-measure $\geqslant \tau - 2$ for all $i$. Then $w_1, \ldots, w_k, \ldots$ are linearly independent over $k$ $\mod \llangle \Rel\rrangle$ (i.e., their cosets in $k\Fr / \llangle \Rel\rrangle$ are linearly independent).
\end{lemma}
\begin{proof}
In~\cite{AKPR2} we proved that every non-zero element of $\Ideal = \llangle \Rel\rrangle$ contains a monomial with non-empty virtual chart (see Proposition~8.5 and Theorem~5 in~\cite{AKPR2}). This is in fact one of the basic and crucial consequences of group-like small cancellation conditions. Hence, there are no polynomials in $\llangle \Rel\rrangle$ that are linear combinations of monomials with empty virtual chart only. Since $w_i$ does not contain occurrences of elements from $\Mon$ of $\M$-measure $\geqslant \tau - 2$, we obtain that the virtual chart of $w_i$ is empty. Therefore, any linear combination of $w_1, \ldots, w_k, \ldots$ does not belong to $\llangle \Rel\rrangle$. Thus, their cosets in $k\Fr / \llangle \Rel\rrangle$ are linearly independent.
\end{proof}

\medskip
Now we introduce preliminary statements needed for Section~\ref{non_amenability_sec}.

In~\cite{AKPR2} (Section~6 and Section~7.1) we introduce $f$-characteristic of a monomial: $f(A) = (\mincov(A), \nvirt(A))$, where $\nvirt(A)$ is the number of virtual members of the chart of~$A$ and $\mincov(A)$ is the number of elements in a minimal covering of~$A$. To calculate $\mincov(A)$ we consider all maximal occurrences of $\Mon$ in $A$ and count the minimal number of such occurrences that is required in order to cover all letters of $A$ that belong to $\Mon$ (see the examples below, $\lbrace u_i\rbrace_{i = 1}^{k}$ is the set of maximal occurrences in $A$ enumerated from left to right).
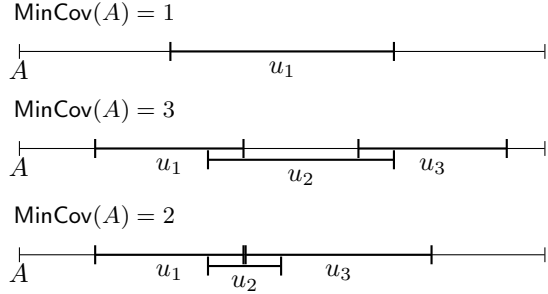
\begin{center}
\begin{tikzpicture}
\draw[|-|] (0, 0) to node[below, at start] {$A$} (7, 0);
\draw[|-|, thick] (2, 0) to node[below, midway] {$u_1$} (5, 0);
\node at (1, 0.5) {\small $\mincov(A) = 1$};
\end{tikzpicture}

\begin{tikzpicture}
\draw[|-|] (0, 0) to node[below, at start] {$A$} (7, 0);
\draw[|-|, thick] (1, 0) to node[below, midway] {$u_1$} (3, 0);
\draw[|-|, thick] (2.5, -0.15) to node[below, midway] {$u_2$} (5, -0.15);
\draw[|-|, thick] (4.5, 0) to node[below, midway] {$u_3$} (6.5, 0);
\node at (1, 0.5) {\small $\mincov(A) = 3$};
\end{tikzpicture}

\begin{tikzpicture}
\draw[|-|] (0, 0) to node[below, at start] {$A$} (7, 0);
\draw[|-|, thick] (1, 0) to node[below, midway] {$u_1$} (3, 0);
\draw[|-|, thick] (2.5, -0.15) to node[below, midway] {$u_2$} (3.5, -0.15);
\draw[|-|, thick] (3, 0) to node[below, midway] {$u_3$} (5.5, 0);
\node at (1, 0.5) {\small $\mincov(A) = 2$};
\end{tikzpicture}
\end{center}

If $p = \sum_{i = 1}^n\alpha_ia_i \in \Rel$, then $a_{i_1}, a_{i_2}$ are called~\emph{incident monomials}. Given a monomial $Z$, we can consecutively replace its virtual members of the chart by incident monomials, then do the same operation with the resulting monomials, etc.. All monomials that are obtained by this process are called~\emph{derived monomials of~$Z$}. Their linear span over $k$ is denoted by $\langle Z\rangle_d$. The subspace of $\langle Z\rangle_d$ generated by all monomials with $f$-characteristics $< f(Z)$ is denoted by $\Low\langle Z\rangle_d$. The set of derived monomials of $Z$ with the same $f$-characteristics as $f(Z)$ is denoted by $\DMUp{Z}$. If $\langle Z_i\rangle_d \neq \langle Z_j\rangle_d$, then $\DMUp{Z_i} \cap \DMUp{Z_j} = \varnothing$ by construction of derived monomials. Let~$\Dp\langle Z\rangle_d \subseteq \langle Z\rangle_d$ be linearly generated by all elements of the form $\sum_{i = 1}^n\alpha_i La_iR$, where $\sum_{i = 1}^n\alpha_ia_i \in \Rel$ and all $La_iR$ are derived monomials of~$Z$.

We also need the following extension of incidence of monomials. Let $Lu_iR \in \Fr$, $u_i \in \Mon$ for $1 \leqslant i \leqslant t$, and all pairs $u_i, u_{i + 1}$ be incident monomials. Assume that $u_i$ is a virtual member of the chart of $Lu_iR$ for $1 < i < t$. Then $u_1, u_t$ are called $Lu_1R$-incident monomials.

The following statement describes the structure of~$k\Fr / \Ideal$ (see~\cite{AKPR2}, Section~9.2, Theorem~3).
\begin{theorem}
\label{sc_ring_basis}
As a linear space
\begin{equation*}
k\Fr / \Ideal  \cong \bigoplus\limits_{j \in J} \langle Z_j\rangle_d / (\Low\langle Z_j\rangle_d + \Dp\langle Z_j\rangle_d),
\end{equation*}
where $\lbrace \langle Z_j\rangle_d\rbrace_{j \in J}$ are all different spaces of such form such that the corresponding quotient spaces in the above sum are non-trivial. Furthermore let $\mathcal{B}_j \subseteq \langle Z_j\rangle_d$ be a set of monomials such that $\mathcal{B}_j + \Low\langle Z_j\rangle_d + \Dp\langle Z_j\rangle_d$ is a basis of the corresponding quotient space. Then $\mathcal{B} + \Ideal$, where $\mathcal{B} = \bigsqcup\limits_{j \in J}\mathcal{B}_j$, is a basis of $k\Fr / \Ideal$.
\end{theorem}

\begin{remark}
\label{f_char_decomp}
In~\cite{AKPR2} we show that $\Fr = \bigsqcup_{i} \DMUp{Z_i}$ for some set of monomials $\lbrace Z_i\rbrace_i$. We can choose a set $\lbrace Z_j \rbrace_{j \in J}$ in Theorem~\ref{sc_ring_basis} as a subset of $\lbrace Z_i\rbrace_i$.
\end{remark}

Based on $f$-characteristics and the partition of~$\Fr$ given in Remark~\ref{f_char_decomp}, one can define a linear order $<_f$ on monomials such that if $f(U_1) < f(U_2)$, then $U_1 <_f U_2$ (see~\cite{AKPR2}, Section~10, Definition~10.1, we compare $f$-characteristics lexicographically).

Let $A$ be a monomial. In~\cite{AKPR2} (Section~$8.2$, Proposition~$8.10$) we show that as a linear space
\begin{equation}
\label{tensor_prod}
\langle A\rangle_d / (\Low\langle A\rangle_d + \Dp\langle A\rangle_d) \cong \bigotimes_{s = 1}^n A_s[A] / (\Low_s[A] + \Dp_s[A]),
\end{equation}
where $n = \nvirt(A)$, $A_s[A]$ is linearly generated by monomials $A$-incident to $s$-th virtual member of the chart of $A$, and $\Dp_s[A]$ is linearly generated by $\Rel \cap A_s[A]$ (see~\cite{AKPR2}, the beginning of Section~$8.2$ for the formal definitions of $A_s[A]$, $\Low_s[A]$, $\Dp_s[A]$; see also Definition~$8.4$ and Lemma~$8.8$ for the explicit description of isomorphism). In particular, formula~\eqref{tensor_prod} holds for $A = Z_j$, where $Z_j$ is a monomial from Theorem~\ref{sc_ring_basis}.

Note that if $a$ is $s$-th virtual member of the chart $Z_j$, then by definition $a \in A_s[Z_j]$. If $Z \in \DMUp{Z_j}$ and $b$ is its $s$-th virtual member of the chart, then by construction the corresponding to $b$ monomial in $A_s[Z_j]$ is of the form $\widehat{b} = c_1\cdot b\cdot c_2$, where $c_1, c_2 \in \Sp$. In particular $\M(b) - 2 \leqslant \M(\widehat{b}) \leqslant \M(b) + 2$.

Let $a + \sum_i\alpha_i a_i \in \Rel$ and $A$ be a monomial of the form $LaR$, then the transition $A \mapsto -\sum_i\alpha_i La_iR$ is called \emph{a multi-turn}. If $A = LaR$ and $Z$ are monomials, $A \in \DMUp{Z}$, $a$ is $s$-th virtual member of the chart of $A$, and $a + \sum_i\alpha_i a_i \in \Dp_s[Z] \subseteq \langle\Rel\rangle$, then the transition $A \mapsto -\sum_i\alpha_i La_iR$ is called \emph{an iterated multi-turn}.

Now we construct a special basis of $k\Fr/ \Ideal$. In every space $A_s[Z_j]$ let us sort the monomials by $<_f$ in ascending order, that is, first we compare their $f$-characteristics, and for monomials with the same $f$-characteristics we compare them using $\deglex$-order. Now we construct a set $\mathcal{B}_{js}$ as follows: we run through this sequence of monomials and add a monomial to the set if its coset in $A_s[Z_j] / (\Low_s[Z_j] + \Dp_s[Z_j])$ is linearly independent with the previously added elements. As a result we obtain a basis of $A_s[Z_j] / (\Low_s[Z_j] + \Dp_s[Z_j])$, since $A_s[Z_j]$ is linearly generated by monomials. Then using the sets $\mathcal{B}_{js}$ together with~\eqref{tensor_prod} and Theorem~\ref{sc_ring_basis}, we obtain the set of monomials $\mathcal{B}$ such that $\mathcal{B} + \Ideal$ a basis of $k\Fr / \Ideal$. We call $\mathcal{B} + \Ideal$ \emph{the minimal basis with respect to $<_f$}. The minimal basis satisfy the following properties:

\begin{lemma}
\label{minimal_basis_decomp}
Let $\mathcal{B} + \Ideal$ be the minimal basis of $k\Fr/ \Ideal$ and $A\notin \mathcal{B}$ be a monomial. Then for the basis decomposition of $A + \Ideal = \sum_{i = 1}^t \alpha_i u_i + \Ideal$, $u_i \in \mathcal{B}$, we have that $u_i <_f A$.
\end{lemma}
\begin{proof}
First assume that $\langle A \rangle_d / (\Low\langle A\rangle_d + \Dp\langle A\rangle_d)$ is trivial. Then $A + \Ideal$ is a linear combination of monomials with $f$-characteristics $< f(A)$. We continue in the same way with these monomials and finally obtain that $A + \Ideal = \sum_{i}\beta_i W_i + \Ideal$, where $\langle W_i \rangle_d / (\Low\langle W_i\rangle_d + \Dp\langle W_i\rangle_d)$ are non-trivial.

So it remains to consider the case when $\langle A \rangle_d / (\Low\langle A\rangle_d + \Dp\langle A\rangle_d)$ is non-trivial. Then $A \in \DMUp{Z_j}$ for one of $Z_j$ from Theorem~\ref{sc_ring_basis}. We obtain that $A + \Ideal = \sum \alpha_i q_i + \sum \gamma_i z_i + \Ideal$, where $q_i \in \mathcal{B}\cap\DMUp{Z_j}$, and $f(z_i) < f(A)$. By construction of minimal basis and definition of $<_f$, we see that $q_i <_f A$, so the claim follows.
\end{proof}

\begin{remark}
\label{sc_strong}
Let $\theta \geqslant \tau + 1$ be a constant. In Section~\ref{non_amenability_sec} we will use stronger version of Small Cancellation Axiom, namely we will require that every element from the linear span of $\Rel$ contains after additive cancellations a monomial of $\M$-measure $\theta$.
\end{remark}

\begin{lemma}
\label{short_basis_elements}
Assume that Small Cancellation Axiom for $\Rel$ holds with $\theta \geqslant \tau + 1$ (cf. Remark~\ref{sc_strong}). Then all monomials in $A_s[Z_j] \setminus \Low_s[Z_j]$ of $\M$-measure $\leqslant \theta - 1$ belong to the corresponding set $\mathcal{B}_{js}$ of the minimal  basis.
\end{lemma}
\begin{proof}
Assume by the contrary, that a monomial $u \in A_s[Z_j] \setminus \Low_s[Z_j]$, $\M(u) \leqslant \theta - 1$, and $u \notin \mathcal{B}_{js}$. Then $u = \sum \alpha_i w_i \mod \Low_s[Z_j] + \Dp_s[Z_j]$, where $w_i \in \mathcal{B}_{js}$. Hence $\sum \alpha_i w_i - u \in \Low_s[Z_j] + \Dp_s[Z_j]$. By Small Cancellation Axiom every polynomial in $\Dp_s[Z_j]$ contains a monomial of $\M$-measure $\geqslant \theta$. Such a monomial can not belong to $\Low_s[Z_j]$, because it always corresponds to a virtual member of the chart. Hence there exists $w_i$ with $\M(w_i) \geqslant \theta$. However, by construction of the minimal basis, $f(w_i) \leqslant f(u)$, a contradiction.
\end{proof}

Let $A$ be a monomial. For the further convenience we describe explicitly the process of the basis decomposition of $A + \Ideal$ with respect to the basis~$\mathcal{B} + \Ideal$:
\begin{itemize}
\item
Assume that $\langle A \rangle_d / (\Low\langle A\rangle_d + \Dp\langle A\rangle_d)$ is trivial. Then it follows from~\eqref{tensor_prod} that there exists~$a$ a virtual member of the chart of $A = LaR$, and $Q = a + \sum_i \alpha_i a_i \in \Dp_s[A]$ such that $f(La_iR) < f(A)$ for all $i$. We take the leftmost such $a$ and make the transition $A = LaR \mapsto -\sum_i \alpha_i La_iR$, and check whether $\langle La_iR \rangle_d / (\Low\langle La_iR\rangle_d + \Dp\langle La_iR\rangle_d)$ are trivial. For the trivial ones we repeat the current step, for non-trivial ones we move to the next step.
\item
Assume that $\langle A \rangle_d / (\Low\langle A\rangle_d + \Dp\langle A\rangle_d)$ is non-trivial. Then $A \in \DMUp{Z_j}$ for one of $Z_j$ from Theorem~\ref{sc_ring_basis}. Then we take $a$ the leftmost virtual member of the chart of~$A$ that requires decomposition and process the corresponding iterated multi-turn of~$A$. For the resulting monomials we return either to the first step, or proceed to the second step.
\end{itemize}
After every iteration of the above process we either decrease $f$-characteristics of monomials, or shift to the right in the chart of monomials with the same $f$-characteristics and do not add new virtual members of the chart in the latter case. Hence the process converges.

\begin{lemma}
\label{derived_monomials_in_decomp}
Assume that Small Cancellation Axiom for $\Rel$ holds with $\theta \geqslant \tau + 1$ (cf. Remark~\ref{sc_strong}). Let $LaR \mapsto  -\sum_i \alpha_i La_iR$ be some step of the basis decomposition stated above of some monomial. Then $\M(a) \geqslant \theta - 2$ and for every~$i$ there exist $c_0, \ldots, c_t \in \Mon$, $c_0 = a$, $c_t = a_i$ ($c_j$ and $t$ depend on $a_i$), such that $c_j, c_{j + 1}$ are incident monomials and $\M(c_j) \geqslant \theta - 2$ for all $0 \leqslant j < t$.
\end{lemma}
\begin{proof}
Let $A = LaR$ and denote the iterated multi-turn $LaR \mapsto  -\sum_i \alpha_i La_iR$ by $\Phi$. Assume by the contrary, that $\M(a) \leqslant \theta - 3$. By Lemma~\ref{short_basis_elements}, the transformation $\Phi$ can not appear when $\langle A \rangle_d / (\Low\langle A\rangle_d + \Dp\langle A\rangle_d)$ is non-trivial. So assume that $\langle A \rangle_d / (\Low\langle A\rangle_d + \Dp\langle A\rangle_d)$ is trivial. Since in this case all $f(La_iR) < f(LaR)$, all $a_i$ have $\M$-measure $\leqslant \tau - 1$. However by Small Cancellation Axiom there exists $a_i$ with $\M(a_i) \geqslant \theta$, a contradiction.

Consider $\sum_j \gamma_j T_j$, where $T_j \in \Rel$ and denote by $\mathcal{T}$ the set of all polynomials $T_j$ that participate in this linear combination. By Small Cancellation Axiom every $T_j$ contains monomials with $\M$-measure $\geqslant \theta$. Consider a partition of $\mathcal{T} = \bigsqcup_l \mathcal{T}_l$ as follows: if two polynomials from $\mathcal{T}$ belong to different sets, then they do not have common monomials of $\M$-measure $\geqslant \theta$. We fix the smallest possible such partition (with the smallest possible sets, see Remark~\ref{partition_rem}). Then the construction implies that if $z_1, z_2$ are different monomials that participate in polynomials form the same set $\mathcal{T}_l$, then there exist $u_0, \ldots, u_t \in \Mon$, $u_0 = z_1$, $u_t = z_2$, such that $u_j, u_{j + 1}$ are incident monomials and $\M(u_j) \geqslant \theta$ for all $0 < j < t$.

By definition $a + \sum_i \alpha_i a_i = \sum_j \gamma_j T_j$, where $T_j \in \Rel$ and pairwise different. By Small Cancellation Axiom every element in the linear span of $\mathcal{T}_l$ contains a monomial with $\M$-measure $\geqslant \theta$. Therefore if $f(La_iR) < f(LaR)$ for all $i$, then there exits only one set $\mathcal{T}_l$ in the above partition. Hence if $\langle A \rangle_d / (\Low\langle A\rangle_d + \Dp\langle A\rangle_d)$ is trivial, we are done.

Now assume that $\langle A \rangle_d / (\Low\langle A\rangle_d + \Dp\langle A\rangle_d)$ is non-trivial. Then $A \in \DMUp{Z_m}$ for one of $Z_m$ from Theorem~\ref{sc_ring_basis}. Let $\widetilde{a}, \widetilde{a}_i$ be elements of $A_s[Z_m]$ that correspond to $a, a_i$, respectively. Then $\widetilde{a} + \sum_i \alpha_i \widetilde{a}_i \in \Dp_s[Z_m]$. Hence $\widetilde{a} + \sum_i \alpha_i \widetilde{a}_i = \sum_j \gamma_j T_j$, where $T_j \in \Rel$. By Lemma~\ref{short_basis_elements}, we have that $\M(\widetilde{a}) \geqslant \theta$. Hence $\widetilde{a}$ appear in polynomials only in one set $\mathcal{T}_{l_0}$ of the corresponding partition. We again apply Lemma~\ref{short_basis_elements} and see that all monomials that are not cancelled in the linear combination $\sum_j \gamma_j T_j$ and participate in the polynomials from $\mathcal{T}_{l_0}$ belong to $\mathcal{B}_{ms}$. Therefore the partition consists of the single set~$\mathcal{T}_{l_0}$. We have that $a + \sum_i \alpha_i a_i = \sum_j \gamma_j c_1\cdot T_j \cdot c_2$, where $c_1, c_2$ are small pieces. Thus we are done.
\end{proof}

\begin{remark}
\label{partition_rem}
The formal definition of the partition $\mathcal{T} = \bigsqcup_l \mathcal{T}_l$ in fact says the following. We take a polynomial $T_1 \in \mathcal{T}$. If there exist any polynomials in $\mathcal{T}$ that have common monomials with $T_1$ of $\M$-measure $\geqslant \theta$, we join them with $T_1$ to one preparatory set $\mathcal{T}_1'$. If in $\mathcal{T} \setminus \mathcal{T}_1'$ there exist polynomials that have common monomials with some elements from $\mathcal{T}_1'$ of $\M$-measure $\geqslant \theta$, we join them with $\mathcal{T}_1'$. We repeat the process while it is possible and as a final result obtain the set $\mathcal{T}_1$. If $\mathcal{T} \setminus \mathcal{T}_1 \neq \varnothing$, we take a polynomial from there and repeat the above process for it, etc.. The final result is the required partition.
\end{remark}

\section{Construction of a free subalgebra}

Recall that the set of small pieces~$\Sp$ is closed under taking subwords. If $\Sp^{\prime} \supseteq \Sp$ is some extension of the set of small pieces such that $\Sp^{\prime}$ is closed under taking subwords, then in the same way as $\M$-measure defined with respect to $\Sp$, one can define $\M^{\prime}$-measure with respect to $\Sp^{\prime}$ and state Small Cancellation Axiom and Isolation Axiom with respect to $\M^{\prime}$-measure. If $\Rel$ satisfies these new axioms (still together with Compatibility Axiom), then the small cancellation theory for rings works with $\M^{\prime}$-measure as well as it works with $\M$-measure. In particular, Lemma~\ref{empty_chart_words} holds with $\M^{\prime}$-measure.

From now on we fix a special~$\Sp^{\prime}$ which consists of $\Sp$ and all letters that belong to $\Mon$. Then clearly~$\Sp^{\prime}$ is closed under taking subwords, and in what follows we use $\M^{\prime}$-measure with respect to $\Sp^{\prime}$.

\begin{lemma}
\label{small_accumulation}
Let $w_1, w_2$ be cyclically reduced words such that $w_1, w_2 \notin \Mon$, $w_1\cdot w_2$ and $w_2\cdot w_1$ have no cancellations. Assume that $w_1, w_2$ do not contain maximal occurrences of $\Mp$-measure $> \mu$. Then every word of the form $w_1^{\alpha_1}w_2^{\beta_1}\ldots$, where $\alpha_i, \beta_i$ are non-negative integers, does not contain maximal occurrences of $\Mp$-measure $> 2\mu$.
\end{lemma}
\begin{proof}
Let $u$ be a maximal occurrence in $w_1^{\alpha_1}w_2^{\beta_1}\ldots$, $\alpha_i, \beta_i \geqslant 0$. Since $w_1, w_2 \notin \Mon$ and $\Mon$ is closed under taking subwords, we obtain that $u$ is properly contained either in $w_1^2$, or in $w_2^2$, or in $w_1w_2$, or in $w_2w_1$. Hence, $\Mp(u) \leqslant 2\mu$.
\end{proof}

\begin{lemma}
\label{small_increasing}
Let $w \notin \Mon$. Let $x$ be a single letter such that either $x \in \Sp'$, or $x\notin \Mon$. Assume that all maximal occurrences in $w$ have $\Mp$-measure $\leqslant \mu$. Then all maximal occurrences in $x\cdot w$ and $w\cdot x$ have $\Mp$-measure $\leqslant \mu + 1$. Moreover, if $x\cdot w$ have no cancellations, then $xw \notin \Mon$; if $w\cdot x$ have no cancellations, then $wx \notin \Mon$.
\end{lemma}
\begin{proof}
If $x$ prolongs some maximal occurrence, it can prolong it by at most one small piece, since $x \in \Sp'$, so the first part follows.

The second part follows from the fact that $\Mon$ is closed under taking subwords.
\end{proof}

\begin{remark}
Clearly, in order to have a free subalgebra in the quotient ring $k\Fr / \llangle \Rel\rrangle$, one needs to avoid having a very big set of relators $\Rel$. The conditions of Theorem~\ref{free_subelgebra_exists} provide one possible to do that.
\end{remark}

\begin{lemma}
\label{free_subalgebra_generators}
Let $\Fr$ be a non-cyclic free group. Let $\Rel \subseteq k\Fr$ satisfy Compatibility Axiom. Assume moreover that there exist non-empty reduced words that do not belong to $\Mon$, and $\Rel$ does not contain relations of the form $u = 0$, where $u$ is a monomial. Let $\Sp^{\prime}$ be the union of $\Sp$ and all letters that belong to $\Mon$. Then there exist two words $w_1, w_2 \notin \Mon$ such that they are cyclically reduced, there are no cancellations in $w_1\cdot w_2$, $w_2\cdot w_1$, the words $w_1, w_2$ do not contain occurrences of $\Mp$-measure $> 4$ and do not have common suffix. Furthermore we always can choose $w_1, w_2$ such that they contain some occurrences from $\Mon$.
\end{lemma}
\begin{proof}
Let $w$ be a non-empty reduced word that do not belong to $\Mon$. First assume that $w$ contains some occurrences from $\Mon$. Then there exists a subword $w'$ of $w$ such that $w' \notin \Mon$, it contains some occurrences of $\Mon$, and does not contain occurrences of $\Mp$-measure $> 2$. Indeed, a possible overlap of two different maximal occurrences is a small piece. Hence, considering the set of all maximal occurrences in $w$ (which is non-empty by the initial assumption), one can easily find such a subword. See the picture below ($\lbrace u_i\rbrace_{i = 1}^{k}$ is the set of all maximal occurrences in $w$ enumerated from left to right).

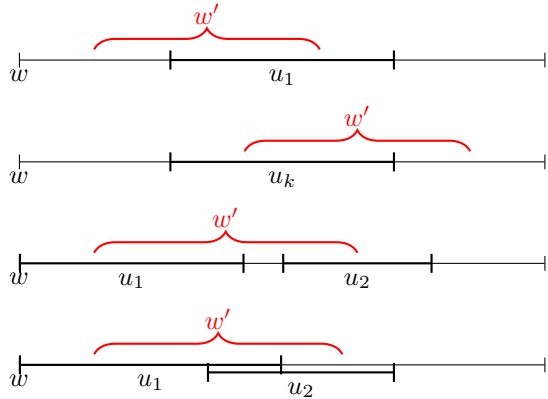
\begin{center}
\begin{tikzpicture}
\draw[|-|] (0, 0) to node[below, at start] {$w$} (7, 0);
\draw[|-|, thick] (2, 0) to node[below, midway] {$u_1$} (5, 0);
\draw [thick, decorate, decoration={brace, amplitude=8pt, raise=4pt}, red] (1, 0) to node[midway, above, yshift=10pt] {$w'$} (4, 0);
\end{tikzpicture}

\begin{tikzpicture}
\draw[|-|] (0, 0) to node[below, at start] {$w$} (7, 0);
\draw[|-|, thick] (2, 0) to node[below, midway] {$u_k$} (5, 0);
\draw [thick, decorate, decoration={brace, amplitude=8pt, raise=4pt}, red] (3, 0) to node[midway, above, yshift=10pt] {$w'$} (6, 0);
\end{tikzpicture}

\begin{tikzpicture}
\draw[|-|] (0, 0) to node[below, at start] {$w$} (7, 0);
\draw[|-|, thick] (0, 0) to node[below, midway] {$u_1$} (3, 0);
\draw[|-|, thick] (3.5, 0) to node[below, midway] {$u_2$} (5.5, 0);
\draw [thick, decorate, decoration={brace, amplitude=8pt, raise=4pt}, red] (1, 0) to node[midway, above, yshift=10pt] {$w'$} (4.5, 0);
\end{tikzpicture}

\begin{tikzpicture}
\draw[|-|] (0, 0) to node[below, at start] {$w$} (7, 0);
\draw[|-|, thick] (0, 0) to node[below, midway] {$u_1$} (3.5, 0);
\draw[|-|, thick] (2.5, -0.1) to node[below, midway] {$u_2$} (5, -0.1);
\draw [thick, decorate, decoration={brace, amplitude=8pt, raise=4pt}, red] (1, 0) to node[midway, above, yshift=10pt] {$w'$} (4.3, 0);
\end{tikzpicture}
\end{center}
If $w'$ is not cyclically reduced, then let $z$ be a letter such that $w'z$ is reduced and cyclically reduced. Such letter exists because $\Fr$ has at least~$2$ free generators. Since $\Sp^{\prime}$ contains all letters from $\Mon$, we have that either $z\notin \Mon$, or $z \in \Sp'$. Hence by Lemma~\ref{small_increasing} the word $w'z$ does not contain maximal occurrences of $\Mp$-measure $> 3$ and $w'z \notin \Mon$. So further we assume that $w'$ is cyclically reduced and does not contain maximal occurrences of $\Mp$-measure $> 3$.

Since $w'$ is cyclically reduced, there are two possibilities: $w' = a\widetilde{w}a$, or $w' = a_1\widetilde{w}a_2$, where $a, a_1, a_2$ are letters, $a_1 \neq a_2^{\pm 1}$. In the first case let $x$ be a letter different from $a^{\pm 1}$, in the second case $x = a_1$. Such letter exists because $\Fr$ has at least~$2$ free generators. Then $w'\cdot x$ has no cancellations and $w'x$ is cyclically reduced. In the same way as above we obtain that $w'x$ does not contain maximal occurrences of $\Mp$-measure $> 4$. By construction $w'x\cdot w'$ and $w'\cdot w'x$ have no cancellations and $w', w'x$ have no common suffix, so we put $w_1 = w'$ and $w_2 = w'x$.

Now assume that $w$ does not contain any occurrences from $\Mon$. Then the powers of $w$ does not contain occurrences from $\Mon$ as well, because $\Mon$ is closed under taking subwords. Let $\Rel \ni p = \sum_{i = 1}^n \alpha_ia_i$ and $M$ be a reduced power of $w$ such that $\abs{M} > 10\abs{a_i}$ for all $i$ and consider $M \cdot p$. Let $M = M_1M_2$, where $M_2$ is the longest cancellation among $M \cdot a_i$. Then $M_1\neq 1$. By Compatibility Axiom $M_2\cdot p = \sum_{i = 1}^n \alpha_ib_i \in \Rel$ and $M_1b_i$ have no cancellations for all $i$. Since $p$ contains at least two different monomials, there exists $b_{i_0} \neq 1$. Hence $M_1b_{i_0} \notin \Mon$ and contains some elements from $\Mon$, so we repeat the above argument for $M_1b_{i_0}$.
\end{proof}

\begin{theorem}
\label{free_subelgebra_exists}
Let $\Fr$ be a non-cyclic free group, $\Rel \subseteq k\Fr$ satisfy Compatibility Axiom and $\mathcal{S}$ be the set of small pieces with respect to $\Rel$. Let $\Sp^{\prime}$ be the union of $\Sp$ and all letters that belong to $\Mon$. Let $\tau \geqslant 11$ and $\Rel$ satisfy satisfy Small Cancellation conditions with respect to $\Mp$-measure. Assume that there exist non-empty reduced words that do not belong to $\Mon$. Then $k\Fr / \llangle \Rel\rrangle$ contains a subalgebra that is isomorphic to free associative algebra with two free generators.
\end{theorem}
\begin{proof}
Notice that Compatibility Axiom together with Small Cancellation Axiom guarantee that there are no relations of the form $u = 0$ in $\Rel$, where $u$ is a monomial. Let $w_1, w_2$ be words that exist by Lemma~\ref{free_subalgebra_generators}. Then by Lemma~\ref{small_accumulation}, every word of the form $w_1^{\alpha_1}w_2^{\beta_1}\ldots$, where $\alpha_i, \beta_i$ are non-negative integers, does not contain maximal occurrences of $\Mp$-measure $> 8$. Since $\tau \geqslant 11$, we see that $9 \leqslant \tau - 2$, therefore $w_1^{\alpha_1}w_2^{\beta_1}\ldots$ does not contain maximal occurrences of $\Mp$-measure $\geqslant \tau - 2$.

Since $w_1, w_2$ have no common suffix by construction, every word \linebreak $W \in \left\lbrace w_1^{\alpha_1}w_2^{\beta_1}\ldots \mid \alpha_i, \beta_i \geqslant 0\right\rbrace$ can be decomposed as a word over $w_1, w_2$ in a unique way. Hence different words over $w_1, w_2$ are different elements of the free group~$\Fr$. Hence Lemma~\ref{empty_chart_words} implies that the set $\left\lbrace w_1^{\alpha_1}w_2^{\beta_1}\ldots + \Ideal \mid \alpha_i, \beta_i \geqslant 0\right\rbrace$, where $\Ideal = \llangle \Rel\rrangle$, is linearly independent over~$k$. Thus $w_1 + \Ideal, w_2 + \Ideal$ generate the free associative algebra over~$k$.
\end{proof}

\section{Non-amenability of small cancellation rings}
\label{non_amenability_sec}

We use the definition of non-amenability due to L.\,Bartholdi (see~\cite{B}). An algebra $R$ over field $k$ is called \emph{non-amenable} if there exist $\varepsilon > 0$ and a finite-dimensional space $S \subseteq R$ such that for every finite-dimensional space $H \subseteq R$ there exists $s \in S$ such that
\begin{equation*}
\frac{\dim_k((H + sH) / H)}{\dim_k(H)} \geqslant \varepsilon.
\end{equation*}

\begin{theorem}
\label{sc_rings_non_amenable}
Let $\Fr$ be a non-cyclic free group, $\Rel \subseteq k\Fr$ satisfy Compatibility Axiom and $\mathcal{S}$ be the set of small pieces with respect to $\Rel$. Let $\Sp^{\prime}$ be the union of $\Sp$ and all letters that belong to $\Mon$. Let $\tau \geqslant 16$. Assume that $\Rel$ satisfy Small Cancellation Axiom with $\tau + 10$ (cf. Remark~\ref{sc_strong}) and Isolation Axiom with $\tau$, both of them hold with respect to $\Mp$-measure. Assume that there exist non-empty reduced words that do not belong to $\Mon$. Assume that if both $u, u^{-1} \in \Mon$, then $\Mp(u) = \Mp(u^{-1})$. Then $k\Fr / \llangle \Rel\rrangle$ is non-amenable.
\end{theorem}
\begin{proof}
We put $\Ideal = \llangle \Rel\rrangle$. Let $\mathcal{B} + \Ideal$ be the minimal basis with respect to $<_f$ of $k\Fr / \Ideal$ (see the construction in Section~\ref{prel}). Let~$w_1, w_2$ be words constructed in Lemma~\ref{free_subalgebra_generators} (so $w_1 + \Ideal, w_2 + \Ideal$ generate the free associative algebra). One can easily see that at least one of $w_1, w_2$ is not a proper power. Without loss of generality assume that $w_1$ is not a proper power. Consider words
\begin{equation*}
v_i = v_i(\gamma_i, \delta_i) = w_2w_1^{\gamma_i}w_2w_1^{\gamma_i + 1}\ldots w_2w_1^{\delta_i}w_2,\ \gamma_i, \delta_i \in \mathbb{Z},\ i = 1, 2, 3,
\end{equation*}
such that $\gamma_1 \geqslant 90$, $\delta_i \geqslant \gamma_i + 90$, $\gamma_{i + 1} \geqslant \delta_i + 90$\footnote{The precise number $90$ is not relevant. Any big enough number that guarantees small common prefixes of different $v_i$ suffice.}. Since the words have a specific form, given two different words $A_1, A_2 \in \lbrace v_1^{\pm 1}, v_2^{\pm 1}, v_3^{\pm 1}\rbrace$, their common prefix $c$ satisfy $\abs{c} < \frac{1}{30}\abs{A_j}$, $j = 1, 2$. We fix a linear space
\begin{equation*}
S = \langle v_1, v_2, v_3\rangle + \Ideal \subseteq k\Fr / \Ideal.
\end{equation*}

Let $H \subseteq k\Fr / \Ideal$ be a finite dimensional subspace with basis $z_1 + \Ideal, \ldots, z_t + \Ideal$, where  $z_i \in k\Fr$ (not necessarily monomials). Let $z_i + \Ideal$ be a linear combination of some $u_{ij} + \Ideal$, $u_{ij}\in \mathcal{B}$, with non-zero coefficients. Without loss of generality we assume that $u_{ij}$ are ordered in descending order with respect to $<_f$ and that the corresponding matrix of coefficients of $z_1 + \Ideal, \ldots, z_t + \Ideal$ is upper-triangular (one can apply Gauss elimination process if necessary). Let $u_i$ be the highest monomial of $z_i$. Then the set $\lbrace z_1, \ldots, z_t\rbrace$ can be partitioned as follows:
\begin{enumerate}[label=H\,\arabic*.]
\item
$u_i$ has a common prefix with $v_1^{\pm 1}$ of length $\geqslant \frac{1}{3}\abs{v_1}$;
\item
$u_i$ has a common prefix with $v_2^{\pm 1}$ of length $\geqslant \frac{1}{3}\abs{v_2}$;
\item
$u_i$ has a common prefix with $v_3^{\pm 1}$ of length $\geqslant \frac{1}{3}\abs{v_3}$;
\item
the remaining elements.
\end{enumerate}
Since common prefixes of $v_1^{\pm 1}, v_2^{\pm 1}, v_3^{\pm 1}$ are relatively small, the above sets are indeed pairwise disjoint.

Among the sets H1---H3 there exists at least one set that contains $\leqslant \frac{1}{3}t$ elements (in particular it can be empty). Without loss of generality assume that it is H1. Then $\abs{\mathrm{H}2 \sqcup \mathrm{H}3 \sqcup \mathrm{H}4} \geqslant \frac{2}{3}t$. To shorten the notations we put $v_1 = v$ and without loss of generality assume that $\mathrm{H}2 \sqcup \mathrm{H}3 \sqcup \mathrm{H}4 = \lbrace z_1, \ldots, z_l\rbrace$ (still with upper-triangular matrix of basis coefficients). By construction $l \geqslant \frac{2}{3}t$.

It is sufficient to show that $(\langle z_1, \ldots, z_l \rangle +\Ideal) \cap (v\langle z_1, \ldots, z_l\rangle + \Ideal) = \lbrace 0 + \Ideal\rbrace$. Indeed, suppose that this property holds. Since $z_1, \ldots, z_l$ are linearly independent $\mod \Ideal$ and $v$ is invertible, we obtain that $v\cdot z_1, \ldots, v\cdot z_l$ are linearly independent $\mod \Ideal$. Hence $z_1, \ldots, z_l, v\cdot z_1, \ldots, v\cdot z_l$ are linearly independent $\mod \Ideal$. Since $l \geqslant \frac{2}{3}t$, we have that $\dim_k(H + vH) \geqslant \frac{4}{3}t$ and $\dim_k((H + vH) / H) \geqslant \frac{1}{3}t$. So, the definition of non-amenability holds with $\varepsilon = \frac{1}{3}$ and finite-dimensional space~$S$.
\medskip

The proof of $(\langle z_1, \ldots, z_l \rangle +\Ideal) \cap (v\langle z_1, \ldots, z_l\rangle + \Ideal) = \lbrace 0 + \Ideal\rbrace$ is organized as several steps.
\medskip

\textbf{Step 1.} Let $A_1, A_2$ be monomials such that $A_1 <_f A_2$, the maximal cancellation in $v\cdot A_1$ has length $> \frac{1}{2}\abs{v}$, the maximal cancellation in $v\cdot A_2$ has length $< \frac{1}{3}\abs{v}$. Then $\mincov(v\cdot A_1) + 1013 < \mincov(v\cdot A_2)$.

Roughly speaking, this holds because long subwords of $v$ have big value of $\mincov$. In more precision, since $f(A_1) \leqslant f(A_2)$, we have that $\mincov(A_1) \leqslant \mincov(A_2)$. Let $v\cdot A_1 = v'A_1'$, $v\cdot A_2 = v''A_2'$.  Since $w_1, w_2$ contain occurrences from $\Mon$ and $w_1, w_2 \notin \Mon$, we obtain that $\mincov(v')$ is significantly smaller than $\mincov(v'')$, for instance $\mincov(v') + 1015 <  \mincov(v'')$.

Let $v = v'c' = v''c''$. Then $A_1 = c'^{-1}A_1'$, $A_2 = c''^{-1}A_2'$ and $c''$ is a subword of $c'$. For every word of the form $W_1W_2$ we have that
\begin{equation*}
\mincov(W_1) + \mincov(W_2) - 1 \leqslant \mincov(W_1W_2) \leqslant \mincov(W_1) + \mincov(W_2),
\end{equation*}
and $\mincov$ does not increase when we take a subword. Therefore
\begin{align*}
\mincov(A_1') &\leqslant \mincov(A_1) - \mincov(c'^{-1}) + 1 \leqslant\\
&\leqslant \mincov(A_2) - \mincov(c''^{-1}) + 1 \leqslant  \mincov(A_2') + 1.
\end{align*}
Combining these observations with the above inequality, we obtain the result.
\medskip

\textbf{Step 2.} Let $w \in \mathcal{B}$. We will show that there exists a decomposition of $v\cdot w$ as follows: $v\cdot w + \Ideal = \sum_j \beta_jb_j + \sum_t \gamma_tC_t + \Ideal$, where $b_j \in \mathcal{B}$ with $\mincov(b_j) = \mincov(v\cdot w)$, $C_t$ are monomials such that $\mincov(C_t) < \mincov(v\cdot w)$, and $v^{-1}\cdot C_t <_f w$.

Consider the basis decomposition process stated at the end of Section~\ref{prel}. Let us do it for $v\cdot w$ with the following modification: if we obtain a monomial $U$ with $\mincov(U) < \mincov(v\cdot w)$ during the process, we stop to decompose it. Then we obtain that $v\cdot w + \Ideal = \sum_j \beta_jb_j + \sum_t \gamma_tC_t + \Ideal$, where $b_j \in \mathcal{B}$ with $\mincov(b_j) = \mincov(v\cdot w)$, and $C_t$ are monomials such that $\mincov(C_t) < \mincov(v\cdot w)$. Let us show that $v^{-1}\cdot C_t <_f w$. Note that in the argument for Step~2 we use notations that are independent from the previously used ones (except the notation for~$v$).

Let $v \cdot w = v'w'$. By construction we have that $b_j$ and $C_t$ are derived monomials of $v'w'$. Furthermore, by Lemma~\ref{derived_monomials_in_decomp}, it is possible to obtain them from $v\cdot w$ replacing only virtual members of the chart of $\Mp$-measure $\geqslant \tau + 8$. Recall that if $a$ is a virtual member of the chart of a monomial $LaR$ and $a, a_1$ are incident monomials, then $f(La_1R) \leqslant f(LaR)$ (see~\cite{AKPR2}, Lemma~7.1), where $f(LaR) = (\mincov(LaR), \nvirt(LaR))$. Therefore when we construct $C_t$, the derived monomials that arise in the process (except $C_t$) have the same $\mincov$ as $\mincov(v'w')$. Let $B \mapsto C_t$ be the last replacement in the construction of $C_t$, so $B$ and all the derived monomials that appear in between $v'w'$ and $B$ have $\mincov$ equal to $\mincov(v'w')$. The latter implies that every virtual member of the chart that is replaced in the process correspond to a unique virtual member of the chart of $v'w'$ (see Definition~6.2 of an admissible replacement, Definition~6.5 of virtual members of the chart in~\cite{AKPR2} and Remark~6.2). Hence we can mark virtual members of the chart of $v'w'$ that are replaced, and then do all the replacements in the order from right to left. By Lemma~\ref{derived_monomials_in_decomp} in the latter procedure we replace virtual members of the chart of $\Mp$-measure $\geqslant \tau + 8$ as well (see~\cite{AKPR2}, Proposition~6.17).

Consider in detail the sequence of replacements that starts from $B$ end results in $C_t$ and the replacements are done from right to left. Namely, let $u_1, \ldots, u_n$ be virtual members of the chart of $v'w'$ that correspond to the replacing ones. Notice that there are no $u_i$ that are contained in $v'$, since maximal occurrences in $v'$ are of $\Mp$-measure $\leqslant 8$. Then at the first step we have a sequence of replacements $u_n = u_n^{(0)} \mapsto u_n^{(1)}\mapsto \ldots \mapsto u_n^{(K_n)}$ such that $u_n^{(i)}, u_n^{(i + 1)}$ are incident monomials and $\Mp(u_n^{(i)}) \geqslant \tau + 8$ for $0 \leqslant i < K_n$. At the second step we do the similar replacements for the occurrence that corresponds to $u_{n - 1}$ in the resulting monomial of the first step, etc.. Denote this sequence by $v'w' = U^{(n)} \mapsto U^{(n - 1)} \mapsto \ldots \mapsto U^{(1)} \mapsto U^{(0)} = B$, where $U^{(s - 1)}$ is the result of $(n - s + 1)$-th step.

Now we multiply the latter sequence by $v^{-1}$ from the left and consider the sequence $w = v^{-1}\cdot U^{(K)} \mapsto v^{-1}\cdot U^{(K - 1)} \mapsto \ldots \mapsto v^{-1}\cdot U^{(1)} \mapsto v^{-1}\cdot U^{(0)} = v^{-1}\cdot B$. First assume that $u_1$ is also a maximal occurrence in $w$. Then every transformation $v^{-1}\cdot U^{(s)} \mapsto v^{-1}\cdot U^{(s - 1)}$ is a replacement of the same elements from~$\Mon$ as in $U^{(s)} \mapsto U^{(s - 1)}$. Therefore $v^{-1}\cdot B$ is a derived monomial of $w$. Now assume that $u_1$ is not a maximal occurrence in $w$. For $s\neq 1$ the transformations $v^{-1}\cdot U^{(s)} \mapsto v^{-1}\cdot U^{(s - 1)}$ and $U^{(s)} \mapsto U^{(s - 1)}$ are still replacements of the same elements from~$\Mon$. Let $\widetilde{u}_1$ be the maximal occurrence in $U^{(1)}$ that corresponds to $u_1$. Then $\widetilde{u}_1 = u_1''\widetilde{u}_1'$, where $u_1''$ is a suffix of $v'$ (possibly empty). Since $\Mp(u_1'') \leqslant 8$, we obtain that $\Mp(\widetilde{u}_1') \geqslant \Mp(\widetilde{u}_1) - 8 \geqslant \tau$. Let $d\widetilde{u}_1'$ be the maximal prolongation of $\widetilde{u}_1'$ in $v^{-1}\cdot U^{(1)}$. Then $\Mp(d\widetilde{u}_1') \geqslant \tau$, so it is a virtual member of the chart of $v^{-1}\cdot U^{(1)}$. Let the replacement $U^{(1)} \mapsto U^{(0)}$ be of the form $\widetilde{u}_1 \mapsto \ldots \mapsto a$ with all intermediate monomials with $\Mp$-measure $\geqslant \tau + 8$. Then the replacement $v^{-1}\cdot U^{(1)} \mapsto v^{-1}\cdot U^{(0)}$ is of the form $d\widetilde{u}_1' \mapsto \ldots \mapsto d{u_1''}^{-1}\cdot a$. Notice that $d{u_1''}^{-1}$ is a subword of $v^{-1}$, therefore, given a monomial $X$, we have that $\Mp(d{u_1''}^{-1}\cdot X) \geqslant \Mp(X) - 8$. So all the intermediate monomials in $du_1' \mapsto \ldots \mapsto d{u_1''}^{-1}\cdot a$ are of $\Mp$-measure $\geqslant \tau$, hence $v^{-1}\cdot B$ is a derived monomial of $w$.

Now we study the replacement $B \mapsto C_t$, $\mincov(C_t) < \mincov(B)$. Let $B = LbR$, $C_t = Lb_1R$, where $b, b_1$ are incident monomials. Then $\Mp(b) \geqslant \tau + 8$ and $\Mp(b_1) \leqslant 2$ (see~\cite{AKPR2}, Lemma~6.1). Let $w = zw'$. The above argument implies that either $B = v'D$, $v^{-1}\cdot B = zD$, or $B = v''aD$ and $v^{-1}\cdot B = z'\widetilde{a}D$, where $v''$ is a prefix of $v$, $z'$ is a prefix of~$z$, and $\widetilde{a} = d{u_1''}^{-1}\cdot a$. Looking at the position of the replaced occurrence in $B$, we obtain that $v^{-1}\cdot B \mapsto v^{-1}\cdot C_t$ is a replacement of a virtual member of the chart of $v^{-1}\cdot B$. So $v^{-1}\cdot C_t$ is a derived monomial of $v^{-1}\cdot B$ and hence of $w$. In particular $f(v^{-1}\cdot C_t) \leqslant f(v^{-1}\cdot B)$. If $f(v^{-1}\cdot B) < f(w)$, then $f(v^{-1}\cdot C_t) < f(w)$, so $v^{-1}\cdot C_t <_f w$ and we are done.

Now we shall consider the case $f(v^{-1}\cdot B) = f(w)$, that is, every transformation $v^{-1} \cdot U^{(s)} \mapsto v^{-1} \cdot U^{(s - 1)}$ is a replacement of a virtual member of the chart by a virtual member of the chart. Recall from above that either $B = v'D$, $v^{-1}\cdot B = zD$ (where $w = zw'$), or $B = v''aD$ and $v^{-1}\cdot B = z'\widetilde{a}D$, where $v''$ is a prefix of $v$, $z'$ is a prefix of $z$, and $\widetilde{a} = d{u_1''}^{-1}\cdot a$. First assume that $b$ have no common part with $v''$ and $b \neq a$. Since $\mincov(B) = \mincov(v'w')$, the occurrence $a$ is not contained in $b$. Therefore $b = b''b'$, where $b'$ is a common part with $D$ and $\Mp(b'') \leqslant 1$ (i.e., $b''$ is either empty, or a small piece). The corresponding maximal occurrence in $v^{-1}\cdot B$ is of the form $cb'$, where $\Mp(c) \leqslant 1$ ($c$ can be empty), because $f(v^{-1}\cdot B) = f(w)$, so $\widetilde{a}$ can not be covered by $c$. Hence the replacement in $v^{-1}\cdot B$ is of the form $cb' \mapsto c\cdot {b''}^{-1} \cdot b_1$. We have that $\Mp(c\cdot {b''}^{-1} \cdot b_1) \leqslant 1 + 1 + 2 = 4$. Hence $c\cdot {b''}^{-1} \cdot b_1$ is not a virtual member of the chart. So $f(v^{-1}\cdot C_t) < f(v^{-1}\cdot B) = f(w)$ and the result follows.

Having $f(v^{-1}\cdot B) = f(w)$, it remains to consider the following two possibilities: the first one is $B = v'D$, $v^{-1}\cdot B = zD$, $b$ have a common part with $v'$, and second one is $B = v''aD$, $v^{-1}\cdot B = z'\widetilde{a}D$, $b = a$. We study only the latter case, the first one is similar. In this case we make the step back and glue the replacements $U^{(1)} \mapsto B \mapsto C_t$. So we have the replacement $\widetilde{u}_1 \mapsto \ldots \mapsto b_1$ in $U^{(1)}$ and the corresponding replacement $d\widetilde{u}_1' \mapsto \ldots \mapsto d{u_1''}^{-1} \cdot b_1$ in $v^{-1}\cdot U^{(1)}$ with the result $v^{-1}\cdot C_t$. We obtain that $\Mp(d{u_1''}^{-1} \cdot b_1) \leqslant \Mp(d) + \Mp({u_1''}^{-1} \cdot b_1) \leqslant \Mp(d) + 10$. Using $\tau \geqslant 16$, we see that $\Mp(d{u_1''}^{-1} \cdot b_1) \leqslant \Mp(d) + \tau - 6 \leqslant \Mp(d\widetilde{u}_1') - 5$. Therefore $v^{-1}\cdot C_t <_f w$ by the definition of $<_f$ (see~\cite{AKPR2}, Definition~10.1). This completes the proof of Step~2.
\medskip

\textbf{Step 3.} Let $w_1, \ldots, w_n \in \mathcal{B}$ be pairwise different monomials such that $\mincov(v\cdot w_i) = m$, and the cancellations in $v\cdot w_i$ has length $\leqslant \frac{1}{2} \abs{v}$ for all $1 \leqslant i \leqslant n$. In the basis decomposition of a non-trivial linear combination $\sum_i \alpha_i v\cdot w_i + \Ideal$ all monomials have $\mincov \leqslant m$ by the property of the minimal basis (see Lemma~\ref{minimal_basis_decomp}). We will show that in this basis decomposition after additive cancellations there exist monomials with $\mincov$ equal to~$m$. Furthermore, all of them have a common prefix with $v$ of length $> \frac{1}{3}\abs{v}$. In particular, the highest monomial with respect to $<_f$ is like this.

Without loss of generality we suppose that $w_1 >_f w_2 >_f\ldots >_f w_n$ and  $\alpha_i \neq 0$ for all $1\leqslant i \leqslant n$. Step~$2$ implies that we can write $v\cdot w_i = \sum_j \beta_j^{(i)} b_j^{(i)} + \sum_t\gamma_t^{(i)}C_t^{(i)} + \Ideal$, where $b_j^{(i)} \in \mathcal{B}$, $\mincov\left(b_j^{(i)}\right) = m$ for all $i, j$, the monomials $C_t^{(i)}$ satisfy $\mincov(C_t^{(i)}) < m$ and $v^{-1}\cdot C_t^{(i)} <_f w_i \leqslant_f w_1$. By Lemma~\ref{minimal_basis_decomp} we obtain that in the basis decomposition of $\sum_t\gamma_t^{(i)}C_t^{(i)} + \Ideal$ all basis monomials are with $\mincov < m$.

Now assume by the contrary, that in the basis decomposition of $\sum_i \alpha_i v\cdot w_i + \Ideal$ after additive cancellations only  monomials with $\mincov < m$ survive. Then it follows from the above that the monomials $b_j^{(i)}$ are cancelled for all $i, j$. Since $v^{-1}\cdot C_t^{(i)} <_f w_i \leqslant_f w_1$, multiplying $\sum_i \alpha_i v\cdot w_i + \Ideal$ by $v^{-1}$ from the left, we obtain the decomposition of $\sum_i \alpha_i  w_i + \Ideal$ that consists of monomials that are $<_f w_1$. So when we make the basis decomposition of this element, there can not appear $w_1$ by Lemma~\ref{minimal_basis_decomp}. This contradicts to the uniqueness of the basis decomposition $\sum_i \alpha_i  w_i + \Ideal$, since all $w_i \in \mathcal{B}$. So some of the elements $b_j^{(i)}$ are not cancelled in the basis decomposition of $\sum_i \alpha_i v\cdot w_i + \Ideal$. To complete the proof of Step~$3$ we notice that all $b_j^{(i)}$ by construction (cf. Step~2 and the assumption that the cancellations in $v\cdot w_i$ has length $\leqslant \frac{1}{2} \abs{v}$) have a common prefix with $v$ of length $> \frac{1}{3}\abs{v}$.
\medskip

\textbf{Step 4.}
We consider a combination $\sum_{i = 1}^l \alpha_i z_i$ (where not all coefficients are zero). Without loss of generality we can assume that $\alpha_1 \neq 0$. Let $\mathcal{U} = \lbrace u_{ij}\rbrace_{ij}$ be the set of all monomials from $\mathcal{B}$ that are not additively cancelled in this linear combination. Recall that $u_i$ is the highest monomial of $z_i$. Since the matrix of basis coefficients of $z_1, \ldots, z_l$ was assumed to be upper-triangular, $u_1$ is not additively cancelled and is the highest monomial with respect to $<_f$ in $\mathcal{U}$. Let $\mathcal{A}_1 \subseteq \mathcal{U}$ be such monomials that the cancellations in $v\cdot u_{ij}$ has length $< \frac{1}{3}\abs{v}$, $\mathcal{A}_2 \subseteq \mathcal{U}$ be such that the cancellations in $v\cdot u_{ij}$ has length in range $[\frac{1}{3}\abs{v}, \frac{1}{2}\abs{v}]$, and $\mathcal{A}_3 \subseteq \mathcal{U}$ be such that the cancellations in $v\cdot u_{ij}$ has length $> \frac{1}{2}\abs{v}$.

We consider monomials of the form $v\cdot u_{ij}$, $u_{ij}\in \mathcal{U}$, with the maximal $\mincov$ among all such monomials. Denote this value by~$m$. Since $u_1 \in \mathcal{A}_1$, Step~1 implies that monomials $v\cdot u_{ij}$ with the maximal $\mincov$ satisfy $u_{ij} \in \mathcal{A}_1 \cup \mathcal{A}_2$. Applying Step~3 to these monomials, we obtain that the highest monomial in the basis decomposition of $\sum_{i = 1}^l \alpha_i v\cdot z_i + \Ideal$ (after additive cancellations) has $\mincov$ equal to~$m$ and has a common prefix with $v$ of length $> \frac{1}{3}\abs{v}$. On the other hand, in every non-trivial linear combination $\sum_j\beta_j z_j$ the highest monomial is equal to some $u_j$, since the matrix of basis coefficients is upper-triangular. By choice of~$v$, a possible common prefix of $u_j$ and $v$ has length $< \frac{1}{3}\abs{v}$. So, the highest monomials in the basis decompositions of $\sum_{i = 1}^l \alpha_i v\cdot z_i + \Ideal$ and $\sum_j\beta_j z_j + \Ideal$ are different, hence $\sum_{i = 1}^l \alpha_i v\cdot z_i + \Ideal \neq \sum_j\beta_j z_j + \Ideal$. This completes the proof.
\end{proof}

\section{Examples}
Theorem~\ref{sc_rings_non_amenable} says in fact that small cancellation rings introduced in \cite{AKPR2} provide us with a lot of new examples of non-amenable rings. For more specific examples we can look at the particular ring constructed in~\cite{AKPR1}. With large enough parameters $\alpha, \beta$ for a special word $V = yx^{\alpha}\ldots yx^{\beta}$, which is used in order to invert $1 + w$, Theorem~\ref{sc_rings_non_amenable} implies that the ring $k\Fr / \llangle 1 + w - V^{-1} \rrangle$ is non-amenable.

Note that for small cancellation groups satisfying $C(4)$ together with $T(4)$, or $C(7)$, non-amenability of the corresponding group algebras follows from the work~\cite{B}, since these groups contain $\Fr_2$. In~\cite{AKPR2} we prove that a group algebra of small cancellation group satisfying the condition $C(22)$ is a small cancellation ring. Using Theorem~\ref{sc_rings_non_amenable}, one can also see that if the corresponding group satisfies the stronger condition $C(52)$, then its group algebra is non-amenable.

\end{document}